\documentclass[11pt,reqno]{amsart}

\usepackage[a4paper,
            top=30mm,bottom=30mm,
            left=35mm,right=35mm]{geometry}

\usepackage{amsmath, amsthm, amsfonts, amssymb, amscd, mathtools, bm} 
\usepackage{xcolor} 
\usepackage[unicode=true]{hyperref} 
\usepackage[shortlabels]{enumitem} 
\usepackage[T1]{fontenc}
\usepackage{empheq}
\usepackage{placeins}

\theoremstyle{plain}
\newtheorem{theorem}{Theorem}[section]
\newtheorem{lemma}[theorem]{Lemma}

\newtheorem{proposition}[theorem]{Proposition}

\theoremstyle{definition}

\theoremstyle{remark}

\numberwithin{equation}{section}

\newcommand{\eps}{\varepsilon}

\newcommand{\dx}{\mathrm{d} x}

\newcommand{\N}{\mathbb{N}}
\newcommand{\R}{\mathbb{R}}

\newcommand{\dm}{\mathrm{d}\mu}
\newcommand{\F}{\mathrm{F}}

\title[]{\boldmath A truncation criterion for compactness \\ in asymptotic $L_p$ spaces}

\author[]{Nuno J. Alves}

\address[N. J. Alves]{
      CEMSE Division, King Abdullah University of Science and Technology, Thuwal, 23955-6900, Saudi Arabia.}
\email{nuno.januarioalves@kaust.edu.sa}

\begin{document}

\begin{abstract}
We prove a compactness criterion for asymptotic $L_p$ spaces over arbitrary measure spaces. Total boundedness is characterized by almost equiboundedness together with total boundedness in $L_p$ of all truncations. As a consequence,
we obtain a simpler proof of the Kolmogorov--Riesz compactness theorem for asymptotic $L_p$ spaces on $\mathbb R^n$.
\end{abstract}

\subjclass[2020]{Primary 46A50, 46E30; Secondary 28A20}
\keywords{Asymptotic~$L_p$ spaces, compactness criterion, total boundedness, truncations, $\F$-spaces, almost equiboundedness, spaces of measurable functions}
\maketitle
\thispagestyle{empty} 


\section{Introduction}

The asymptotic $L_p$ space over a measure space $(X,\Sigma,\mu)$ is defined as
\begin{equation} \label{eq:Lambda_space}
\Lambda^p(X)=\left\{ f:X\to\R \text{ measurable} \ \Big| \ 
\int_X \min(|f|,1)^p \, \dm < \infty\right\},
\end{equation}
where $1\leq p<\infty$ and functions are identified if they agree almost
everywhere. The space $\Lambda^p(X)$ is naturally endowed with the topology
generated by the functional
\begin{equation} \label{eq:F-norm}
f\mapsto \|\!\min(|f|,1)\|_p.
\end{equation}
Here and throughout, $\|\cdot \|_p$ denotes the norm of the standard Lebesgue space $L^p(X)$.

The functional in~\eqref{eq:F-norm} is an $\mathrm F$-norm: it satisfies the usual properties of a norm except
for homogeneity, which is replaced by the following two conditions:
\[
\|\!\min(|\lambda f|,1) \|_p \leq \|\!\min(| f|,1) \|_p,
\]
for all $|\lambda| \leq 1$ and all $f \in \Lambda^p(X)$, and  
\[
\lim_{\lambda \to 0} \|\!\min(|\lambda f|,1) \|_{p} = 0,
\]
for all $f \in \Lambda^p(X)$; see \cite{alves2025F}. 
With this topology, $\Lambda^p(X)$ is an $\mathrm F$-space, that is, a complete
metrizable topological vector space. We recall the proof of this fact in
Section~\ref{sec:asymptotic}. For background on metric linear spaces and
$\mathrm F$-spaces, see~\cite{kalton1984space, rolewicz1985metric}.

The spaces $\Lambda^p(X)$ were introduced in~\cite{alves2025F} as spaces of
functions almost in $L_p$, equipped with the topology of asymptotic
$L_p$-convergence; see also~\cite{alves2024mode,alves2024relation}. A real-valued measurable function $f$ is said to be almost in $L_p$ if, for every $\delta>0$, there exists a measurable set $E_\delta\subseteq X$ such that
\[
\mu(E_\delta)<\delta
\qquad \text{and} \qquad
f\chi_{E_\delta^c}\in L^p(X),
\]
where $E_\delta^c = X \setminus E_\delta$. Related almost-$L_p$ spaces had previously appeared in connection with factorization
and representation questions in functional analysis; see~\cite{galdames2012optimal,calabuig2019representation}. The equivalence between
the original almost-$L_p$ definition and~\eqref{eq:Lambda_space} is given
in Proposition~\ref{prop:charact}.

Although the notation suggests a close relationship with $L^p(X)$, the spaces
$\Lambda^p(X)$ are fundamentally different from Lebesgue spaces. For example,
it was shown in~\cite{alves2025F} that, in the Euclidean setting, $\Lambda^p(\R^n)$ is neither locally convex nor locally bounded, and that its continuous dual is trivial. More generally, on nonatomic measure spaces, the
same type of behavior for $\Lambda^p(X)$ can be deduced from the general theory of metric linear spaces; see~\cite{rolewicz1985metric}.

The main result of this note is an abstract compactness criterion for $\Lambda^p(X)$ over arbitrary measure spaces. It characterizes total boundedness in terms of almost equiboundedness and total boundedness in $L^p(X)$ of all truncations; see~Theorem~\ref{thm:main}. In the Euclidean case, this criterion gives, by applying the classical Kolmogorov--Riesz theorem in $L^p(\R^n)$ to the truncated families, a short proof of the Kolmogorov--Riesz compactness theorem
in $\Lambda^p(\R^n)$ obtained in~\cite{alves2026kolmogorov}. This Euclidean
compactness theorem has also been used to deduce a compactness result of Rellich--Kondrachov type in $\Lambda^p(\R^n)$, which in turn led to a well-posedness theory for $p$-Schr\"odinger equations with integrable data and
confinement in measure; see~\cite{alves2026theory}.

Compactness criteria in function spaces have been extensively studied in many different settings. For the classical Kolmogorov--Riesz theorem in $L^p(\R^n)$, see~\cite{hanche2010kolmogorov, hanche2019improvement}. Related criteria have been obtained for spaces of measurable functions~\cite{krotov2012criteria}, for variable exponent and variable summability spaces~\cite{bandaliyev2018relatively, dymek2023compactness, gorka2015almost}, for Banach and quasi-Banach function spaces~\cite{caetano2016compactness, gorka2016arzela, gorka2019banach, guo2020relatively}, and in other settings related to Riesz--Kolmogorov compactness criteria~\cite{georgescu2004riesz}. The present paper thus fits into this line of results by providing an abstract compactness criterion for asymptotic $L_p$ spaces over arbitrary measure spaces, reducing total boundedness in $\Lambda^p(X)$ to total boundedness of the truncated families in $L^p(X)$ together with a uniform control of large values.

The paper is organized as follows. In Section~\ref{sec:asymptotic}, we recall the basic structure of~$\Lambda^p(X)$, prove that it is an $\mathrm F$-space under the truncation definition, and establish the equivalence between this definition and the original almost-$L_p$ formulation from~\cite{alves2025F}. Section~\ref{sec:compactness} contains the main compactness criterion. In Section~\ref{sec:kolmogorov}, we recover the
Kolmogorov--Riesz compactness theorem in $\Lambda^p(\R^n)$. Finally,
Section~\ref{sec:vitali} gives a short proof, based on the truncation definition, of the Vitali convergence theorem in~$\Lambda^p(X)$ obtained in~\cite{alves2025F}.

\section{Asymptotic $L_p$ spaces} \label{sec:asymptotic}

In this section we collect the basic structural properties of $\Lambda^p(X)$. Although these facts are elementary, we include the proofs in order to make clear that the truncation definition~\eqref{eq:Lambda_space} gives a complete metrizable topological vector space. We also recall the relation between this definition and the original formulation from~\cite{alves2025F}.

We first show that $\Lambda^p(X)$ is an $\mathrm F$-space. In the terminology of Kalton, Peck, and Roberts~\cite{kalton1984space}, this means that it is a metrizable topological vector space which is complete with respect to a translation-invariant metric. In the present case, the metric is given by
\[
(f,g)\mapsto \|\!\min(|f-g|,1)\|_p.
\]

\begin{theorem}\label{thm:Lambda_F-space}
Let $(X,\Sigma,\mu)$ be a measure space and let $1\le p<\infty$. Then $\Lambda^p(X)$ is an $\mathrm{F}$-space. Moreover, when $\mu(X) < \infty$, the space $\Lambda^p(X)$ consists of all real-valued measurable functions on~$X$ with the topology of convergence in measure.
\end{theorem}

The final assertion in Theorem~\ref{thm:Lambda_F-space} should be understood in two parts. If $\mu(X)<\infty$, then every real-valued measurable function belongs to $\Lambda^p(X)$, since \[\min(|f|,1)^p\leq 1.\] The statement about the topology uses the fact that, on finite measure spaces, convergence with respect to the $\F$-norm~\eqref{eq:F-norm} is equivalent to convergence in measure; see~\cite{alves2024mode, alves2024relation, alves2025F}.

That $\Lambda^p(X)$ is a vector space follows immediately from the following two inequalities
\begin{equation} \label{eq:ineq_scalar_prod}
\min(|\lambda a|,1)^p \leq \max\{|\lambda|^p,1\}  \min(|a|,1)^p,
\end{equation}
and
\begin{equation} \label{eq:ineq_triang}
\min(|a+b|,1)^p \leq 2^{p-1}\big( \min(|a|,1)^p + \min(|b|,1)^p \big),
\end{equation}
valid for all $\lambda,a,b \in \R$ and $1 \leq p < \infty$.

From the triangle inequality it also follows that the operation of addition is continuous with respect to the metric of~$\Lambda^p(X)$. In the next lemma we show that the operation of scalar multiplication is also continuous in this
topology, thus establishing that~$\Lambda^p(X)$ is a metrizable topological vector space.

\begin{lemma}
Assume that $\{\lambda_k\}_{k \in \N}$ is a sequence of real numbers converging to a real number $\lambda$, and assume that $\{f_k \}_{k \in \N}$ is a sequence of functions in $\Lambda^p(X)$ converging to a function $f \in \Lambda^p(X)$. Then 
\[\lambda_k \, f_k \to \lambda \, f \qquad \text{in } \Lambda^p(X) \qquad \text{as } k \to \infty. \]
\end{lemma}
\begin{proof}
Since $\lambda_k\to\lambda$, there exists $C>0$ such that
$|\lambda_k|\le C$ for every $k\in\N$. Using the triangle inequality and~\eqref{eq:ineq_scalar_prod} we deduce
\[
\begin{aligned}
\|\min(|\lambda_k f_k-\lambda f|,1)\|_p
&\le
\|\min(|\lambda_k(f_k-f)|,1)\|_p
+
\|\min(|(\lambda_k-\lambda)f|,1)\|_p \\
&\le
\max\{C,1\}\|\min(|f_k-f|,1)\|_p
+
\|\min(|(\lambda_k-\lambda)f|,1)\|_p .
\end{aligned}
\]
The first term tends to $0$, since $f_k\to f$ in $\Lambda^p(X)$.

For the second term, we have
\[
\min(|(\lambda_k-\lambda)f(x)|,1)^p\to0
\]
for almost every $x\in X$. Moreover, for all sufficiently large $k$,
$|\lambda_k-\lambda|\le1$, and hence
\[
\min(|(\lambda_k-\lambda)f|,1)^p
\le
\min(|f|,1)^p.
\]
Since $f\in\Lambda^p(X)$, the right-hand side is integrable. Therefore, by the
dominated convergence theorem,
\[
\|\!\min(|(\lambda_k-\lambda)f|,1)\|_p\to0.
\]

Consequently,
\[
\|\!\min(|\lambda_k f_k-\lambda f|,1)\|_p\to0,
\]
which means that $\lambda_k f_k\to\lambda f$ in $\Lambda^p(X)$.
\end{proof}

To conclude the proof of Theorem~\ref{thm:Lambda_F-space} it remains to prove that the metric of $\Lambda^p(X)$ is complete. This is the content of the next lemma; the standard proof is included for convenience.

\begin{lemma}
Let $\{f_k \}_{k \in \N}$ be a Cauchy sequence in $\Lambda^p(X)$. Then, there exists $f \in \Lambda^p(X)$ such that 
\[f_k \to  f \qquad \text{in } \Lambda^p(X) \qquad \text{as } k \to \infty. \]
\end{lemma}
\begin{proof}
Since $\{f_k\}_{k\in\N}$ is Cauchy in $\Lambda^p(X)$, it is Cauchy in
measure. Indeed, for $0<\alpha<1$,
\[
\alpha^p\mu(\{|f_k-f_\ell|>\alpha\})
\le
\int_X\min(|f_k-f_\ell|,1)^p\,\dm,
\]
and the right-hand side tends to $0$ as $k,\ell\to\infty$.

Hence, there exist a subsequence $\{f_{k_j}\}_{j\in\N}$ and a real-valued measurable
function~$f$ such that
\[
f_{k_j}\to f
\qquad\text{a.e.\ on }X.
\]

We show that this subsequence converges to $f$ in $\Lambda^p(X)$. Let
$\eps>0$. Since $\{f_k\}_{k\in\N}$ is Cauchy in $\Lambda^p(X)$, there exists
$N\in\N$ such that
\[
\|\!\min(|f_k-f_\ell|,1)\|_p<\eps
\]
whenever $k,\ell\ge N$. Fix $j$ such that $k_j\ge N$. Since
$f_{k_\ell}\to f$ a.e., Fatou's lemma gives
\[
\begin{aligned}
\int_X\min(|f_{k_j}-f|,1)^p\,\dm
&\le
\liminf_{\ell\to\infty}
\int_X\min(|f_{k_j}-f_{k_\ell}|,1)^p\,\dm \\
&\le
\eps^p.
\end{aligned}
\]
Thus
\[
\|\!\min(|f_{k_j}-f|,1)\|_p\le \eps
\]
for all sufficiently large $j$. Thus
\[
\|\!\min(|f_{k_j}-f|,1)\|_p\le \eps
\]
for all sufficiently large $j$. Choosing one such $j$, the triangle inequality
gives $f\in\Lambda^p(X)$. Hence
\[
f_{k_j}\to f
\qquad\text{in }\Lambda^p(X).
\]
Since the original sequence is Cauchy in $\Lambda^p(X)$ and has a subsequence converging to~$f$, the whole sequence converges to~$f$ in $\Lambda^p(X)$, again by the triangle inequality.
\end{proof}

We next give the characterization that explains the equivalence between the original definition of $\Lambda^p(X)$ in~\cite{alves2025F} and the one adopted in the present paper.

\begin{proposition} \label{prop:charact}
Fix $1\le p<\infty$, and let $f$ be a real-valued measurable function on $X$.
The following statements are equivalent:
\begin{enumerate}[(i)]
\item $f$ is almost in $L_p$.
\smallskip
\item There exists a sequence $\{f_k\}_{k\in\N}$ in $L^p(X)$ such that
\[
\|\!\min(|f_k-f|,1)\|_p\to0
\qquad\text{as } k\to\infty.
\]

\item $\min(|f|,1)\in L^p(X)$.
\end{enumerate}
\end{proposition}

\begin{proof} We first show that (i) implies (ii). For each $k\in\N$, choose
$E_k\in\Sigma$ such that $\mu(E_k)<1/k$ and
\[
f_k:=f\chi_{E_k^c}\in L^p(X).
\]
Then
\begin{align*}
\|\!\min(|f_k-f|,1)\|_p^p
&=
\int_{E_k}\min(|f\chi_{E_k^c}-f|,1)^p\,\dm  \le
\mu(E_k)
<
\frac1k \to 0,
\end{align*}
as $k \to \infty$.

Next, we show that (ii) implies (iii). Choose $N\in\N$ such that
\[
\|\!\min(|f_N-f|,1)\|_p<1.
\]
Since
\[
\min(|f|,1)
\le
\min(|f-f_N|,1)+\min(|f_N|,1),
\]
the triangle inequality gives
\[
\|\!\min(|f|,1)\|_p
\le
\|\!\min(|f-f_N|,1)\|_p+\|\!\min(|f_N|,1)\|_p
<
1+\|f_N\|_p
<
\infty.
\]
Hence $\min(|f|,1)\in L^p(X)$. 

Finally, we prove that (iii) implies (i). For each $k\in\N$, set
\[
E_k:=\{|f|>k\}.
\]
Then $E_{k+1}\subseteq E_k$ for every $k\in\N$, and
$\mu(E_1)<\infty$, because
\[
\chi_{E_1}\le \min(|f|,1)^p\in L^1(X).
\]
Moreover, the countable intersection
\[
\bigcap_{k\in\N}E_k
\]
has measure zero, since $f$ is real-valued almost everywhere. Hence
\[
\mu(E_k)\to0
\qquad\text{as } k\to\infty.
\]
Given $\delta>0$, choose $K_\delta\in\N$ such that
\[
\mu(E_{K_\delta})<\delta.
\]
On $E_{K_\delta}^c$ we have $|f|\le K_\delta$, and therefore
\[
|f|^p
\le
K_\delta^p\min(|f|,1)^p.
\]
It follows that
\[
f\chi_{E_{K_\delta}^c}\in L^p(X).
\]
Thus $f$ is almost in $L_p$.
\end{proof}

We finish this section with two elementary observations. First, the spaces $\Lambda^p(X)$ are nested as $p$ increases. Second, for bounded functions, belonging to $\Lambda^p(X)$ is the same as belonging to $L^p(X)$.

\begin{proposition}\label{prop:nesting}
Let $1\le p\le q<\infty$. Then
\[
\Lambda^p(X)\subseteq \Lambda^q(X).
\]
\end{proposition}

\begin{proof}
The result follows from the fact that, for $1 \leq p \leq q < \infty$, one has
\[
\min(|a|,1)^q\le \min(|a|,1)^p,
\]
for every $a \in \R$.
\end{proof}

\begin{proposition}\label{prop:bounded-part}
Let $M>0$, and let $f$ be a measurable function on $X$ such that
$|f|\le M$ almost everywhere. Then
\[
f\in \Lambda^p(X)
\qquad\text{if and only if}\qquad
f\in L^p(X).
\]
Consequently,
\[
L^\infty(X)\cap\Lambda^p(X)=L^\infty(X)\cap L^p(X).
\]
\end{proposition}

\begin{proof}
From the hypothesis $|f| \leq M$ a.e., we deduce
\begin{equation}
\|\!\min(|f|,1)\|_p
\le
\|f\|_p
\le
\max\{1,M\}\|\!\min(|f|,1)\|_p,
\end{equation}
from which the result clearly follows.
\end{proof}

\section{Compactness criterion} \label{sec:compactness}

We now state and prove the main compactness criterion of the paper. It says that total
boundedness of a family in $\Lambda^p(X)$ can be characterized by two
conditions: the total boundedness in $L_p$ of each truncated family, and the
uniform control of large values, expressed through approximation by
truncations. 

For $M>0$, we define
\begin{equation} \label{eq:truncations}
T_M(a):=\max\big\{\!-M,\min\{a,M\}\big\}, \qquad a\in\R,
\end{equation}
and write $T_Mf:=T_M\circ f$. 

\begin{theorem} \label{thm:main}
Let $(X,\Sigma,\mu)$ be a measure space, let $1\leq p<\infty$, and let $\mathcal F\subseteq\Lambda^p(X)$. Then $\mathcal F$ is totally bounded in $\Lambda^p(X)$ if and only if the following two conditions hold:
\begin{enumerate}[(i)]
\item For every fixed $M>0$, the truncated family
\[
T_M(\mathcal F):=\big\{T_Mf:\,f\in\mathcal F\big\}
\]
is totally bounded in $L^p(X)$.
\smallskip
\item The family $\mathcal F$ is uniformly approximable by truncations, that is,
\begin{equation} \label{eq:approx_trunc}
\lim_{M\to\infty}\sup_{f\in\mathcal F} \|\!\min(|f-T_M f|,1)\|_p=0.
\end{equation}
\end{enumerate}
\end{theorem}

\medskip
\subsection*{Truncation estimates}

The first
lemma contains the basic properties of $T_M$ that will be used below.

\begin{lemma} \label{lemma_basic}
Let $M>0$. The following statements hold.
\begin{enumerate}[(i)]
\smallskip
\item If $f\in\Lambda^p(X)$, then $T_Mf\in L^p(X)$.
\smallskip
\item For every $f,g\in\Lambda^p(X)$,
\[
 \|\!\min(|T_M f -T_M g|,1) \|_p \leq \|\!\min(|f -g|,1) \|_p.
\]
\item If $u,v\in L^p(X)$ and $|u|,|v|\leq M$ a.e., then
\[
 \|\!\min(|u -v|,1) \|_p\leq \|u-v\|_p\leq C_M \|\!\min(|u -v|,1) \|_p,
\]
where $C_M=\max\{1,2M\}$.
\end{enumerate}
\end{lemma}

\begin{proof}
For (i), observe that
\[
 |T_M(a)|\leq \max\{M,1\}\min(|a|,1), \qquad a\in\R.
\]
Thus $|T_Mf|^p\leq \max\{1,M\}^p\min(|f|,1)^p\in L^1(X)$.

\par

For (ii), the scalar truncation map $T_M:\R\to\R$ is $1$-Lipschitz. Hence
\[
\min(|T_M(f)-T_M(g)|,1)\leq \min(|f-g|,1) \qquad \text{a.e.},
\]
and the claim follows.

\par

Finally, if $|u|,|v|\leq M$, then $|u-v|\leq 2M$. Therefore
\[
 \min(|u-v|,1)\leq |u-v|\leq C_M\min(|u-v|,1) \qquad \text{a.e.},
\]
and taking $L^p$-norms gives (iii).
\end{proof}

The next estimate relates truncation errors to the measure of level sets. In particular,
truncations approximate each fixed element of $\Lambda^p(X)$.

\begin{lemma} \label{lem:truncation_conv}
Let $f\in\Lambda^p(X)$. Then, for every $M>0$,
\begin{equation} \label{eq_truncation_measure}
\mu(|f|>M+1)
\leq \|\!\min(| f -T_M f|,1) \|_p^p
\leq \mu(|f|>M).
\end{equation}
Consequently,
\[
\|\!\min(| f -T_M f|,1) \|_p\to0
\qquad \text{as } M\to\infty .
\]
\end{lemma}

\begin{proof}
Since
\[
 |f-T_Mf|=(|f|-M)_+,
\]
we have
\[
 \|\!\min(| f -T_M f|,1) \|_p^p
 =
 \int_X \min\big((|f|-M)_+,1\big)^p\,\dm.
\]
The integrand is supported in $\{|f|>M\}$ and is bounded by $1$. Hence
\[
 \|\!\min(| f -T_M f|,1) \|_p^p\leq \mu(|f|>M).
\]
On the other hand, if $|f|>M+1$, then $(|f|-M)_+>1$, and therefore
\[
 \min\big((|f|-M)_+,1\big)^p=1.
\]
Thus
\[
 \mu(|f|>M+1)\leq  \|\!\min(| f -T_M f|,1) \|_p^p.
\]

It remains only to observe that $\mu(|f|>M)\to0$ as $M \to \infty$. For $M\ge1$,
\[
 \chi_{\{|f|>M\}}\leq \min(|f|,1)^p\in L^1(X),
\]
and $\chi_{\{|f|>M\}}\to0$ pointwise a.e.\ as $M\to\infty$. By the dominated convergence theorem,
\[
 \mu(|f|>M)\to0 \qquad \text{as } M \to \infty.
\]
The upper bound in \eqref{eq_truncation_measure} therefore gives
\[
 \|\!\min(| f -T_M f|,1) \|_p\to0 \qquad \text{as } M \to \infty,
\]
as claimed.
\end{proof}

The preceding lemma gives the following useful reformulation of the second
condition in Theorem~\ref{thm:main}.

\begin{lemma}\label{lem:almost_equibounded}
Let $\mathcal F\subseteq\Lambda^p(X)$. Then $\mathcal F$ satisfies
\eqref{eq:approx_trunc} if and only if it is almost equibounded, that is,
\begin{equation} \label{eq:almost_equiboundedness}
\lim_{M\to\infty}\sup_{f\in\mathcal F}\mu(|f|>M)=0.
\end{equation}
\end{lemma}

\begin{proof}
By Lemma~\ref{lem:truncation_conv}, for every $f\in\mathcal F$ and every
$M>0$,
\[
\mu(|f|>M+1)
\le
\|\!\min(|f-T_Mf|,1)\|_p^p
\le
\mu(|f|>M).
\]
Taking suprema over $f\in\mathcal F$ and letting $M\to\infty$ gives the
equivalence.
\end{proof}

We now proceed with the proof of the compactness criterion.

\subsection*{Proof of Theorem~\ref{thm:main}}
\label{sec:proof}

Assume first that $\mathcal F$ is totally bounded in $\Lambda^p(X)$. Fix $M>0$, let
$\varepsilon>0$, and let $C_M =\max\{1,2M\}$. Choose
$f_1,\ldots,f_N\in\Lambda^p(X)$ such that, for every $f\in\mathcal F$, there is
$i\in\{1,\ldots,N\}$ with
\[
\|\!\min(|f-f_i|,1)\|_p<\frac{\varepsilon}{C_M}.
\]
By Lemma~\ref{lemma_basic}(i), $T_Mf_i\in L^p(X)$ for each $i=1,\ldots, N$. Moreover, using
Lemma~\ref{lemma_basic}(iii) and then Lemma~\ref{lemma_basic}(ii), we obtain
\begin{align*}
\|T_Mf-T_Mf_i\|_p & \leq C_M\|\!\min(|T_Mf-T_Mf_i|,1)\|_p \\
& \leq C_M\|\!\min(|f-f_i|,1)\|_p \\
& <\varepsilon,
\end{align*}
which proves that $T_M(\mathcal F)$ is totally bounded in $L^p(X)$.

Next, we show that
\[
\lim_{M\to\infty}\sup_{f\in\mathcal F}
\|\!\min(|f-T_Mf|,1)\|_p=0.
\]
Let $\varepsilon>0$. Choose $f_1,\ldots,f_N\in\Lambda^p(X)$ such that, for every
$f\in\mathcal F$, there is $i\in\{1,\ldots,N\}$ with
\[
\|\!\min(|f-f_i|,1)\|_p<\frac{\varepsilon}{3}.
\]
By Lemma~\ref{lem:truncation_conv}, for each $i=1,\ldots,N$ there exists
$M_i>0$ such that
\[
\|\!\min(|f_i-T_Mf_i|,1)\|_p<\frac{\varepsilon}{3}
\]
for every $M\geq M_i$. Setting
\[
M_0:=\max\{M_1,\ldots,M_N\},
\]
we obtain
\[
\|\!\min(|f_i-T_Mf_i|,1)\|_p<\frac{\varepsilon}{3}
\]
for every $M\geq M_0$ and every $i=1,\ldots,N$. Hence, for $M\geq M_0$ and
$f\in\mathcal F$, choosing $i$ as above gives
\begin{align*}
\|\!\min(|f-T_Mf|,1)\|_p
&\leq
\|\!\min(|f-f_i|,1)\|_p
+\|\!\min(|f_i-T_Mf_i|,1)\|_p  \\
&\qquad
+\|\!\min(|T_Mf_i-T_Mf|,1)\|_p  \\
&\leq
2\|\!\min(|f-f_i|,1)\|_p
+\|\!\min(|f_i-T_Mf_i|,1)\|_p  \\
&<\varepsilon,
\end{align*}
where Lemma~\ref{lemma_basic}(ii) was used in the second inequality. Therefore
\[
\sup_{f\in\mathcal F}\|\!\min(|f-T_Mf|,1)\|_p<\varepsilon
\]
for every $M\geq M_0$, as required.
\par

Conversely, assume that
\[
\lim_{M\to\infty}\sup_{f\in\mathcal F}
\|\!\min(|f-T_Mf|,1)\|_p=0
\]
and that $T_M(\mathcal F)$ is totally bounded in $L^p(X)$ for every $M>0$.
Let $\varepsilon>0$ be given. Choose $M>0$ such that
\[
\sup_{f\in\mathcal F}\|\!\min(|f-T_Mf|,1)\|_p<\frac{\varepsilon}{2}.
\]
Since $T_M(\mathcal F)$ is totally bounded in $L^p(X)$, there exist
$h_1,\ldots,h_N\in L^p(X)$ such that for every $f\in\mathcal F$ there is
$i\in\{1,\ldots,N\}$ with
\[
\|T_Mf-h_i\|_p<\frac{\varepsilon}{2}.
\]
Then
\[
\|\!\min(|T_Mf-h_i|,1)\|_p\leq \|T_Mf-h_i\|_p<\frac{\varepsilon}{2}.
\]
Thus
\begin{align*}
\|\!\min(|f-h_i|,1)\|_p
&\leq \|\!\min(|f-T_Mf|,1)\|_p
   + \|\!\min(|T_Mf-h_i|,1)\|_p <\varepsilon,
\end{align*}
so $\mathcal F$ is totally bounded in $\Lambda^p(X)$.

\qed

\section{Kolmogorov--Riesz compactness} \label{sec:kolmogorov}

We now show that Theorem~\ref{thm:main}, combined with the classical Kolmogorov--Riesz theorem in $L^p(\R^n)$, gives a short proof of the Euclidean compactness criterion for $\Lambda^p(\R^n)$ obtained in~\cite{alves2026kolmogorov} under the original definition of these spaces from~\cite{alves2025F}.

We start by recalling the classical Kolmogorov--Riesz theorem in $L^p(\R^n)$
in the form stated in~\cite{hanche2019improvement}. There, the usual boundedness assumption on the family has been shown to be a consequence of the tail and translation conditions; see also~\cite{hanche2010kolmogorov}.

\begin{theorem}[Kolmogorov--Riesz compactness in $L^p(\R^n)$]
\label{thm:KR-L}
Let $1\le p<\infty$, and let $\mathcal F\subseteq L^p(\R^n)$. Then
$\mathcal F$ is totally bounded in $L^p(\R^n)$ if and only if the
following two conditions hold:
\begin{enumerate}[(i)]
\item $\mathcal F$ satisfies the tail condition
\[
\lim_{R\to\infty}
\sup_{f\in\mathcal F}
\int_{\{|x|>R\}}|f(x)|^p\,\dx=0.
\]

\item $\mathcal F$ satisfies the translation condition
\[
\lim_{|y|\to0}
\sup_{f\in\mathcal F}
\int_{\R^n}|f(x+y)-f(x)|^p\,\dx=0.
\]
\end{enumerate}
\end{theorem}

Combining Theorem~\ref{thm:main} and Theorem~\ref{thm:KR-L}, we recover the
Kolmogorov--Riesz compactness theorem in $\Lambda^p(\R^n)$.

\begin{theorem}[Kolmogorov--Riesz compactness in $\Lambda^p(\R^n)$]
\label{thm:KR-Lambda}
Let $1\le p<\infty$, and let $\mathcal F\subseteq\Lambda^p(\R^n)$. Then
$\mathcal F$ is totally bounded in $\Lambda^p(\R^n)$ if and only if the
following three conditions hold:
\begin{enumerate}[(i)]
\item $\mathcal F$ satisfies the truncated tail condition
\[
\lim_{R\to\infty}
\sup_{f\in\mathcal F}
\int_{\{|x|>R\}}\min(|f(x)|,1)^p\,\dx=0.
\]

\item $\mathcal F$ satisfies the truncated translation condition
\[
\lim_{|y|\to0}
\sup_{f\in\mathcal F}
\int_{\R^n}\min(|f(x+y)-f(x)|,1)^p\,\dx=0.
\]

\item $\mathcal F$ is almost equibounded, that is,
\[
\lim_{M\to\infty}
\sup_{f\in\mathcal F}
\big|\{|f|>M\}\big|=0.
\]
\end{enumerate}
\end{theorem}

By Theorem~\ref{thm:main} and Lemma~\ref{lem:almost_equibounded}, total
boundedness in $\Lambda^p(\R^n)$ is equivalent to almost equiboundedness
together with total boundedness in $L^p(\R^n)$ of every truncated family
$T_M(\mathcal F)$. Thus, to prove Theorem~\ref{thm:KR-Lambda}, it remains to
relate the truncated tail and translation conditions to the total
boundedness in $L_p$ of these truncated families. This is done in the next two lemmas.

\begin{lemma}\label{lem:KR-conditions-imply-truncations}
Let $\mathcal F\subseteq\Lambda^p(\R^n)$ satisfy the truncated tail condition
and the truncated translation condition in Theorem~\ref{thm:KR-Lambda}. Then,
for every $M>0$, the truncated family $T_M(\mathcal F)$ is totally bounded in
$L^p(\R^n)$.
\end{lemma}

\begin{proof}
Fix $M>0$. By Lemma~\ref{lemma_basic}\textup{(i)}, we have
$T_Mf\in L^p(\R^n)$ for every $f\in\mathcal F$. We shall verify the tail and translation hypotheses of Theorem~\ref{thm:KR-L}.

We start with the tail condition. Using the inequality
\[
|T_M(a)|^p\le \max\{M^p,1\}\min(|a|,1)^p,
\qquad a\in\R,
\]
we deduce
\[
\int_{\{|x|>R\}} |T_Mf(x)|^p\,\dx
\le
\max\{M^p,1\}
\int_{\{|x|>R\}}\min(|f(x)|,1)^p\,\dx.
\]
Taking the supremum over $f\in\mathcal F$ and then letting $R\to\infty$, the
right-hand side tends to zero by the truncated tail condition.

Next, since $T_M$ is $1$-Lipschitz and
$|T_M(a)-T_M(b)|\le 2M$, we have
\[
|T_M(a)-T_M(b)|^p
\le
\max\{(2M)^p,1\}\min(|a-b|,1)^p,
\qquad a,b\in\R.
\]
Thus, for every $y\in\R^n$ and every $f\in\mathcal F$,
\begin{align*}
\int_{\R^n}  |& T_Mf(x+y)-T_Mf(x)|^p\,\dx \\
& \leq \max\{(2M)^p,1\}
\int_{\R^n}\min(|f(x+y)-f(x)|,1)^p\,\dx.
\end{align*}
Taking the supremum over $f\in\mathcal F$ and then letting $|y|\to0$, the
right-hand side tends to zero by the truncated translation condition.

Therefore, by Theorem~\ref{thm:KR-L},
$T_M(\mathcal F)$ is totally bounded in $L^p(\R^n)$.
\end{proof}

The converse implication uses almost equiboundedness to pass from total
boundedness of the truncated families back to the truncated tail and translation conditions.

\begin{lemma}\label{lem:truncations-imply-KR-conditions}
Let $\mathcal F\subseteq\Lambda^p(\R^n)$ be almost equibounded, and assume that
$T_M(\mathcal F)$ is totally bounded in $L^p(\R^n)$ for every $M>0$. Then
$\mathcal F$ satisfies the truncated tail condition and the truncated
translation condition in Theorem~\ref{thm:KR-Lambda}.
\end{lemma}

\begin{proof}
Let $\varepsilon>0$. By almost
equiboundedness, we may choose $M>0$ such that
\[
\sup_{f\in\mathcal F}\big|\{|f|>M\}\big|<\frac{\varepsilon}{4}.
\]
Moreover, since $T_M(\mathcal F)$ is totally bounded in $L^p(\R^n)$, by Theorem~\ref{thm:KR-L} there exist $R > 0$ and $r > 0$ such that 
\[
\sup_{f\in\mathcal F}
\int_{|x|>R} |T_Mf(x)|^p\,\dx < \frac{\eps}{2},
\]
and
\[\sup_{f\in\mathcal F}
\int_{\R^n}|T_Mf(x+y)-T_Mf(x)|^p\,\dx < \frac{\eps}{2}, \]
whenever $|y| < r$.

We first prove the truncated tail condition. For every $f\in\mathcal F$ we have:
\begin{align*}
\int_{|x|>R} \min(|f(x)|,1)^p \, \dx & = \int_{\{|x| > R \} \cap \{|f| \leq M \}} \min(|f(x)|,1)^p \, \dx \\
& \quad + \int_{\{|x| > R \} \cap \{|f| > M \}} \min(|f(x)|,1)^p \, \dx \\
& \leq \int_{|x|>R} |T_M(f)|^p \, \dx + \big| \{|f| > M \} \big| \\ 
& < \eps.
\end{align*}

We now prove the truncated translation condition. Fix $f\in\mathcal F$ and $y\in\R^n$ with $|y| < r$. Set
\[
A_{f,y}:=\big\{ |f|>M \big\}
\cup
\big\{|\tau_y f|>M \big\},
\]
where $\tau_y f(x) = f(x+y)$, and note that
\[
|A_{f,y}|
\le
2\big|\{|f|>M\}\big|.
\]
If $x \in \R^n\setminus A_{f,y}$, then both $|f(x)|$ and $|f(x+y)|$ are at most $M$, and
therefore
\[
T_Mf(x)=f(x),
\qquad
T_Mf(x+y)=f(x+y).
\]
Thus
\begin{align*}
\int_{\R^n}\min(|f(x+y)-f(x)|,1)^p\,\dx & = \int_{\R^n \setminus A_{f,y}}\min(|f(x+y)-f(x)|,1)^p\,\dx \\
& \quad + \int_{ A_{f,y}}\min(|f(x+y)-f(x)|,1)^p\,\dx \\
& \le
\int_{\R^n}|T_Mf(x+y)-T_Mf(x)|^p\,\dx
+
|A_{f,y}| \\
& \le
\int_{\R^n}|T_Mf(x+y)-T_Mf(x)|^p\,\dx
+
2 \big| \{|f| > M \} \big| \\
& < \eps.
\end{align*}
Since $f$ is arbitrary, the result follows.
\end{proof}

\begin{proof}[Proof of Theorem~\ref{thm:KR-Lambda}]
Assume first that $\mathcal F$ is totally bounded in $\Lambda^p(\R^n)$. By
Theorem~\ref{thm:main}, the truncated family $T_M(\mathcal F)$ is totally
bounded in $L^p(\R^n)$ for every $M>0$, and $\mathcal F$ satisfies
\eqref{eq:approx_trunc}. Hence, by Lemma~\ref{lem:almost_equibounded},
$\mathcal F$ is almost equibounded. Lemma~\ref{lem:truncations-imply-KR-conditions}
then gives the truncated tail and translation conditions.

Conversely, assume that $\mathcal F$ satisfies the three conditions in
Theorem~\ref{thm:KR-Lambda}. By Lemma~\ref{lem:KR-conditions-imply-truncations},
$T_M(\mathcal F)$ is totally bounded in $L^p(\R^n)$ for every $M>0$. Moreover,
condition \textup{(iii)} and Lemma~\ref{lem:almost_equibounded} imply
\eqref{eq:approx_trunc}. Hence the two conditions of Theorem~\ref{thm:main}
hold, and therefore $\mathcal F$ is totally bounded in $\Lambda^p(\R^n)$.
\end{proof}

\section{Vitali convergence} \label{sec:vitali}

We conclude this note with a simple proof of the Vitali convergence theorem in $\Lambda^p(X)$, obtained in~\cite{alves2025F} using the almost-$L_p$ definition of the spaces. The proof below uses the truncation definition of~$\Lambda^p(X)$ together with the classical Vitali convergence theorem in $L^p(X)$, which we now recall; see~\cite{bartle1995elements}.

\begin{theorem}[Vitali convergence in $L^p(X)$] \label{thm:vitali_L}
Let $(X,\Sigma,\mu)$ be a measure space, let $1\le p<\infty$, let
$\{f_k\}_{k \in \N}\subseteq L^p(X)$, and let $f$ be measurable. Then
$\{f_k\}_{k\in\N}$ converges to $f$ in $L^p(X)$ if and only if the following two conditions hold:
\begin{enumerate}[(i)]
\item $\{f_k\}_{k\in\N}$ converges to $f$ in measure,
\smallskip
\item for every $\varepsilon > 0$ there exist $E_\varepsilon \in \Sigma$ with $\mu(E_\varepsilon) < \infty$ and $\delta_\varepsilon > 0$ such that 
\[\sup_{k \in \N} \int_{E_\varepsilon^c} |f_k |^p \, \mathrm{d}\mu < \varepsilon^p\]
and, if $F \in \Sigma$ and $\mu(F) < \delta_\varepsilon$, then
\[\sup_{k \in \N} \int_{E_\varepsilon \cap F} |f_k |^p \, \mathrm{d}\mu < \varepsilon^p \, .\]
\end{enumerate}
\end{theorem}

In $\Lambda^p(X)$, the boundedness of the truncated integrands removes the absolute continuity condition on sets of small measure. The resulting criterion
is the following.

\begin{theorem}[Vitali convergence in $\Lambda^p(X)$] \label{thm:vitali_Lambda}
Let $(X,\Sigma,\mu)$ be a measure space, let $1\le p<\infty$, let
$\{f_k\}_{k \in \N}\subseteq \Lambda^p(X)$, and let $f$ be measurable. Then
$\{f_k\}_{k\in\N}$ converges to $f$ in $\Lambda^p(X)$ if and only if the following two
conditions hold:
\begin{enumerate}[(i)]
\item $\{f_k\}_{k\in\N}$ converges to $f$ in measure,
\smallskip
\item for every $\varepsilon>0$ there exists a measurable set
$E_\varepsilon\in\Sigma$ with $\mu(E_\varepsilon)<\infty$ such that
\[
\sup_{k \in \N}
\int_{E_\varepsilon^c}\min(|f_k|,1)^p\,\dm
<
\varepsilon^p.
\]
\end{enumerate}
In particular, when $\mu(X) < \infty$, convergence in $\Lambda^p(X)$ and convergence in measure coincide.
\end{theorem}

\begin{proof}
Assume first that $\{f_k\}_{k \in \N}$ converges to $f$ in $\Lambda^p(X)$. Then $\{f_k\}_{k \in \N}$ converges to $f$ in measure and $\{\min(|f_k|,1)\}_{k \in \N}$ converges to $\min(|f|,1)$ in $L^p(X)$. The latter implies condition \textup{(ii)} by Theorem~\ref{thm:vitali_L}.

Conversely, assume \textup{(i)} and \textup{(ii)}. We prove that $\{f_k\}_{k\in\N}$ converges to $f$ in $\Lambda^p(X)$. Let $\eps>0$. By \textup{(ii)}, we
may choose a measurable set $E_\varepsilon$ with
$\mu(E_\varepsilon)<\infty$ such that
\[
\sup_{k \in \N}
\int_{E_\varepsilon^c}\min(|f_k|,1)^p\,\dm
<
\frac{\varepsilon^p}{2^{p}}.
\]
Using a subsequence converging almost everywhere to $f$ and Fatou's lemma, we
also have
\[
\int_{E_\varepsilon^c}\min(|f|,1)^p\,\dm
\le
\frac{\varepsilon^p}{2^{p}}.
\]
Therefore
\begin{align*}
\int_X\min(|f_k-f|,1)^p \, \dm  & = 
\int_{E_\varepsilon^c}\min(|f_k-f|,1)^p \, \dm +
\int_{E_\varepsilon}\min(|f_k-f|,1)^p \, \dm \\
& \leq 2^{p-1}
\int_{E_\varepsilon^c}\min(|f_k|,1)^p\,\dm +2^{p-1}
\int_{E_\varepsilon^c}\min(|f|,1)^p\,\dm \\
& \quad + \int_{E_\varepsilon}\min(|f_k-f|,1)^p \, \dm \\
& < \eps^p + \int_{E_\varepsilon}\min(|f_k-f|,1)^p \, \dm.
\end{align*}
Moreover, since $f_k\to f$ in measure and $\mu(E_\varepsilon)<\infty$, we have
\[
\int_{E_\varepsilon}\min(|f_k-f|,1)^p\,\dm \to0
\qquad \text{as } k \to \infty.
\]
Consequently
\[
\limsup_{k \to \infty} \int_X\min(|f_k-f|,1)^p \, \dm \leq \eps^p.
\]
As $\eps>0$ is arbitrary, convergence in $\Lambda^p(X)$ follows. The final
assertion is obtained by taking $E_\varepsilon=X$ in condition \textup{(ii)}
when $\mu(X)<\infty$.
\end{proof}

\section*{Acknowledgments}
This publication is based upon work supported by King Abdullah University of Science and Technology (KAUST) under Award No. ORFS-CRG12-2024-6430.

\end{document}